\documentclass[12pt]{article}     
\usepackage{amsmath,amssymb,amsthm, amscd, graphicx, verbatim, setspace}
\usepackage{cite}
\usepackage{url}

\theoremstyle{plain}
\newtheorem{thm}{Theorem}
\newtheorem{lem}[thm]{Lemma}
\newtheorem{cor}[thm]{Corollary}
\newtheorem{defn}[thm]{Definition}
\newtheorem{rem}[thm]{Remark}
\newtheorem{asmp}[thm]{Assumption}
\newtheorem{ques}[thm]{Question}
\newtheorem{conj}[thm]{Conjecture}
\title{On $k$-visibility graphs}
\author{Matthew Babbitt \qquad J.T. Geneson \qquad Tanya Khovanova\\
\small Department of Mathematics\\[-0.8ex]
\small MIT\\[-0.8ex]
\small Massachusetts, U.S.A.\\
\small\tt tanya@math.mit.edu\\
\small\tt geneson@math.mit.edu\\
\small\tt mbabbitt@mit.edu
}
\date{}
\begin{document}
\maketitle

\begin{abstract}
We examine several types of visibility graphs in which sightlines can pass through $k$ objects. For $k \geq 1$ we bound the maximum thickness of semi-bar $k$-visibility graphs between $\lceil \frac{2}{3} (k + 1) \rceil$ and $2k$. In addition we show that the maximum number of edges in arc and circle $k$-visibility graphs on $n$ vertices is at most $(k+1)(3n-k-2)$ for $n > 4k+4$ and ${n \choose 2}$ for $n \leq 4k+4$, while the maximum chromatic number is at most $6k+6$. In semi-arc $k$-visibility graphs on $n$ vertices, we show that the maximum number of edges is ${n \choose 2}$ for $n \leq 3k+3$ and at most $(k+1)(2n-\frac{k+2}{2})$ for $n > 3k+3$, while the maximum chromatic number is at most $4k+4$. 
\end{abstract}

\section{Introduction}
Visibility graphs are graphs for which vertices can be drawn as regions so that two regions are visible to each other if and only if there is an edge between their corresponding vertices. In this paper we study bar, semi-bar, arc, circle, and semi-arc visibility graphs. We also study a variant of visibility graphs represented by drawings in which objects are able to see through exactly $k$ other objects for some positive integer $k$. These graphs are known as $k$-visibility graphs. 

Dean \emph{et al.} \cite{AlDeank} previously placed upper bounds on the number of edges, the chromatic number, and the thickness of bar $k$-visibility graphs with $n$ vertices in terms of $k$ and $n$. Felsner and Massow \cite{Felsner} tightened the upper bound on the number of edges of bar $k$-visibility graphs and placed bounds on the number of edges, the thickness, and the chromatic number of semi-bar $k$-visibility graphs. Hartke {\it et al.} \cite{Hartke} found sharp upper bounds on the maximum number of edges in bar $k$-visibility graphs.

Other research has found classes of graphs which can be represented as bar visibility graphs. For example Lin \emph{et al.} \cite{bartri} determined an algorithm for plane triangular graphs $G$ with $n$ vertices which outputs a bar visibility representation of $G$ no wider than $\lfloor \frac{22n-42}{15}\rfloor$ with bar ends on grid points in time $O(n)$. Luccio \emph{et al.} \cite{Luccio} proved any bar visibility graph can be transformed into a planar multigraph with all triangular faces by successively duplicating edges, and furthermore that every graph which can be transformed into a planar multigraph with all triangular faces by successively duplicating edges can be represented as a bar visibility graph. 

In Section~\ref{sec:def} we define each type of $k$-visibility graph considered in this paper. In Section~\ref{sec:upperbounds} we bound the maximum thickness of semi-bar $k$-visibility graphs between $\lceil \frac{2}{3} (k + 1) \rceil$ and $2k$. We also bound the maximum number of edges and the chromatic numbers of arc, circle, and semi-arc $k$-visibility graphs. In Section~\ref{sec:semibar} we show an equation based on skyscraper puzzles for counting the number of edges in any semi-bar $k$-visibility graph. 

After an earlier version of this paper was posted, we found an abstract \cite{ntu}, with no paper, that claimed to bound the maximum thickness of bar $k$-visibility graphs between $\lceil \frac{2k+3}{3} \rceil$ and $3k+3$ and the maximum thickness of semi-bar $k$-visibility graphs between $\lceil \frac{2k+5}{6} \rceil$ and $2k$. None of the proofs in our paper use the results claimed in \cite{ntu}.

\section{Definitions and assumptions}\label{sec:def}
In this section we define the various types of visibility graphs and cover conditions that we assume throughout the paper. 

\begin{defn}
The \emph{thickness} $\Theta(G)$ of a graph $G$ is the minimum number of planar subgraphs whose union is $G$. That is, $\Theta(G)$ is the minimum number of colors needed to color the edges of $G$ such that no two edges with the same color intersect.
\end{defn}

Research on bounding graph thickness is motivated by the problem of efficiently designing very large scale integration (VLSI) circuits. VLSI circuits are built in layers to avoid wire crossings which disrupt signals \cite{thickness}. For each graph $G$ with vertices corresponding to circuit gates and edges corresponding to wires, the thickness of $G$ gives an upper bound on how many layers are needed to build the VLSI circuit without wires crossing in the same layer.

Below are definitions for the different kinds of $k$-visibility graphs. 

\subsection{Bar $k$-visibility graphs}
\emph{Bar visibility graphs} are graphs that have the property that the vertices of the graph correspond to the elements of a given set of horizontal segments, called \emph{bars}, in such a way that two vertices of the graph are adjacent whenever there exists a vertical segment which only intersects the horizontal segments corresponding to those vertices. The set of horizontal segments is said to be the \emph{bar visibility representation} of the graph. 

This type of visibility graph was introduced in the 1980s by Duchet \emph{et al.} \cite{Duchet} and Schlag \emph{et al.} \cite{Schlag} mainly for its applications in the development of VLSI. Figure~\ref{barvisible} gives an example of a set of horizontal line segments and its corresponding bar visibility graph. Sightlines are drawn as dashed lines.

\begin{figure}[h]
\begin{center}
\includegraphics[scale=0.4]{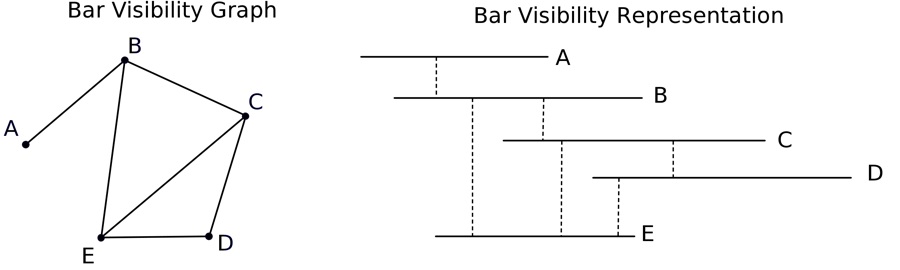}
\setlength{\abovecaptionskip}{0pt}
\caption{A bar visibility graph and its bar visibility representation.}
\label{barvisible}
\end{center}
\end{figure}

Dean \emph{et al.} \cite{AlDeank} extended the definition of bar visibility graphs by allowing visibility through $k$ bars. Vertices $u$ and $v$ have an edge in the graph if and only if there exists a vertical segment intersecting the horizontal bars corresponding to $u$ and $v$ and at most $k$ other horizontal bars. These graphs are called \emph{bar $k$-visibility graphs}. In this terminology visibility means $0$-visibility and we will use these terms interchangeably. Figure~\ref{barkvisible} shows an example of a bar $1$-visibility graph with a bar $1$-visibility representation equivalent to the one in Figure~\ref{barvisible}. Sightlines that pass through an additional bar are drawn thicker than the original sightlines. 

\begin{figure}[h]
\begin{center}
\includegraphics[scale=0.4]{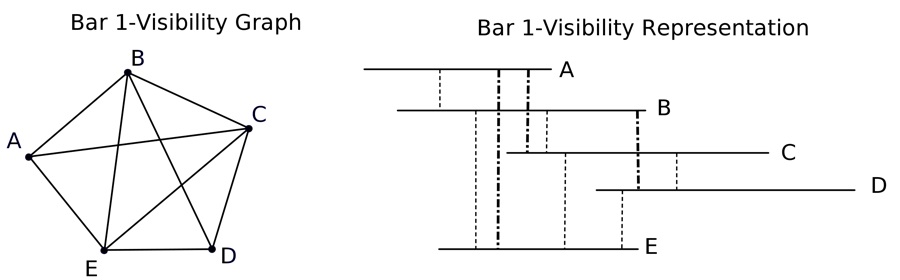}
\setlength{\abovecaptionskip}{0pt}
\caption{A bar $1$-visibility graph and its bar $1$-visibility representation.}
\label{barkvisible}
\end{center}
\end{figure}

\subsection{Semi-bar $k$-visibility graphs}
A \emph{semi-bar $k$-visibility graph} is a bar $k$-visibility graph where the left endpoints of all the bars have $x$-coordinates equal to 0. We prove an upper bound of $2k$ on the thickness of semi-bar $k$-visibility graphs and show that there exist semi-bar visibility graphs with thickness at least $\left\lceil \frac{2}{3}(k+1)\right\rceil$. We will assume that all semi-bars have different lengths unless otherwise specified. If any pair of semi-bars had the same length, then the length of one could be changed without deleting any edges.

Every semi-bar $k$-visibility graph on $n$ vertices, including representations that contain semi-bars of equal lengths, can be represented using semi-bars with integer lengths between $1$ and $n$ inclusive. Therefore every semi-bar $k$-visibility graph can be represented by a sequence of $n$ positive integers between $1$ and $n$ inclusive.

Felsner and Massow \cite{Felsner} proved that the maximum number of edges in a semi-bar $k$-visibility graph with $n$ vertices is $(k+1)(2n-2k-3)$ for $n \geq 2k+2$ and $\binom{n}{2}$ for $n \leq 2k+2$. We give a formula to count the number of edges of an arbitrary semi-bar $k$-visibility graph based on functions of its semi-bar $k$-visibility representation. 

The method we use is inspired by skyscraper problems, which were also examined using permutations in \cite{Kh}. A skyscraper puzzle consists of an empty $n \times n$ grid with numbers written left or right of some rows and above or below some columns. The solver fills the grid with numbers between $1$ and $n$ representing heights of skyscrapers placed in each entry of the grid. The numbers are placed so that no two skyscrapers in the same column or same row have the same height.

If there is a number $m$ above a column in the empty grid, then the numbers $1, \ldots, n$ must be placed in that column so that there are $m$ numbers in the column which are greater than every number above them. If there is a number $m$ below a column in the empty grid, then the numbers $1, \ldots, n$ must be placed in that column so that there are $m$ numbers in the column which are greater than every number below them. The restrictions for the numbers in rows are defined analogously based on the numbers left or right of the row. Then each number $m$ outside the grid corresponds to the number of visible skyscrapers in the row or column adjacent to $m$ which are visible from the location of $m$.

We consider a visibility representation based on skyscraper puzzles in which there is just a single column in which to place numbers. Any such configuration corresponds to a semi-bar visibility graph. Any numbers above (resp. below) the column are the number of semi-bars which are longer than all semi-bars above (resp. below) them.

We show how to count the number of edges in any semi-bar visibility graph by using the numbers above and below the column in its skyscraper configuration. Furthermore we extend the skyscraper analogy to $k$-visibility graphs to show a similar result for semi-bar $k$-visibility representations.

\subsection{Arc, circle, and semi-arc visibility graphs}
An interesting extension of bar visibility graphs is the concept of \emph{arc visibility graphs} introduced by Hutchinson \cite{Hutchinson}, who defined a non-degenerate \emph{cone} in the plane to be a 4-sided region of positive area with two opposite sides being arcs of circles concentric about the origin and the other two sides being (possibly intersecting) radial line segments. Two concentric arcs $a_1$ and $a_2$ are then said to be \emph{radially visible} if there exists a cone that intersects only these two arcs and whose two circular ends are subsets of the two arcs.

A graph is then called an \emph{arc visibility graph} if its vertices can be represented by pairwise disjoint arcs of circles centered at the origin such that two vertices are adjacent in the graph if and only if their corresponding arcs are radially visible. \emph{Circle visibility graphs} are defined in nearly the same way, with the difference that vertices can be represented as circles as well as arcs. Note that all arc visibility graphs are also circle visibility graphs. Figure~\ref{arcvisible} shows an arc visibility graph and its arc visibility representation.

In any arc visibility graph, the arcs can be expressed uniquely as a set of polar coordinates $\{(r_i,\alpha):\alpha_{i,1}\leq \alpha\leq \alpha_{i,2}\}$ such that $r_i$ is positive, $\alpha_{i,1}$ is in the interval $[0,2\pi)$, and $0 \leq \alpha_{i,2}-\alpha_{i,1}<2\pi$. We call the endpoints corresponding to the coordinates $(r_i,\alpha_{i,1})$ and $(r_i,\alpha_{i,2})$ the \emph{negative} and \emph{positive endpoints} of arc $a_i$, respectively. 

\begin{figure}[h]
\begin{center}
\includegraphics[scale=0.5]{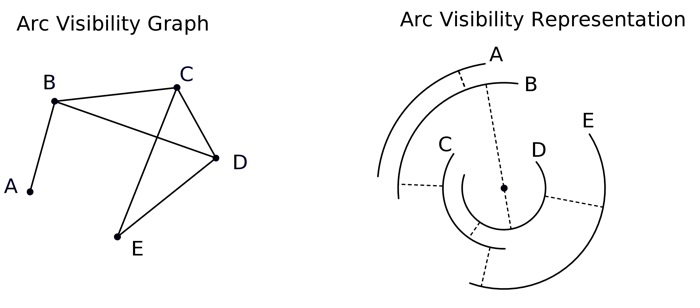}
\setlength{\abovecaptionskip}{0pt}
\caption{An arc visibility graph and its arc visibility representation.}
\label{arcvisible}
\end{center}
\end{figure}

We also examine \emph{arc $k$-visibility graphs} and \emph{circle $k$-visibility graphs}, where cones are allowed to intersect $k$ additional arcs and circles. Figure~\ref{arckvisible} shows the arc $1$-visibility graph of the arc visibility representation shown in Figure~\ref{arcvisible}.

\begin{figure}[h]
\begin{center}
\includegraphics[scale=0.5]{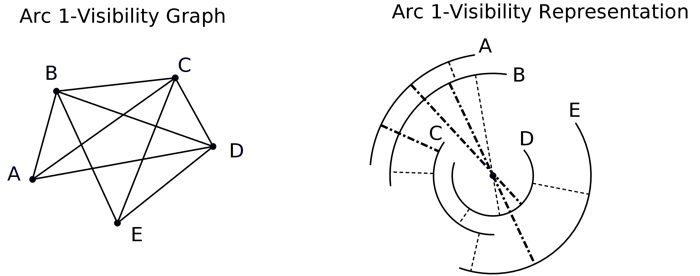}
\setlength{\abovecaptionskip}{0pt}
\caption{An arc $1$-visibility graph and its arc $1$-visibility representation.}
\label{arckvisible}
\end{center}
\end{figure}

We can also define semi-arc $k$-visibility graphs to be arc $k$-visibility graphs in which every arc's negative endpoint lies on the $x$-axis. Like semi-bar $k$-visibility graphs, every semi-arc $k$-visibility graph on $n$ vertices can be represented by a sequence of $n$ integers. Start with a line on the $x$-axis, colored blue on the positive part and red on the negative part, and rotate the line counterclockwise around the origin until it intersects the positive endpoint of a semi-arc. If the blue part of the line intersects the positive endpoints of any semi-arcs, then assign the semi-arcs the integer $1$. If the red part of the line intersects the positive endpoints of any semi-arcs, then assign the semi-arcs the integer $-1$. 

Continue rotating the line counterclockwise. If the blue part of the line intersects the positive endpoints of any semi-arcs, then assign those semi-arcs the least positive integer that is greater in absolute value than all integers assigned when the line was in a previous configuration. If the red part of the line intersects the positive endpoints of any semi-arcs, then assign those semi-arcs the greatest negative integer that is greater in absolute value than all integers assigned when the line was in a previous configuration. After the line has rotated by an angle of $\pi$, all of the semi-arcs are assigned nonzero integers with magnitude at most $n$. 

It follows by definition that all bar $k$-visibility graphs are also arc $k$-visibility graphs and all semi-bar $k$-visibility graphs are also semi-arc $k$-visibility graphs. Indeed let $G$ be a graph on $n$ vertices that has a bar $k$-visibility representation in which the horizontal endpoints of bar $i$ are $a_{i} \geq 0$ and $b_{i} > a_{i}$ and the height of bar $i$ is $h_{i} > 0$. Define $M = \max_{1 \leq i \leq n}\left\{a_{i}, b_{i} \right\}$. For each $i$, draw the arc of radius $h_{i}$ centered at the origin between the angles $\pi\frac{a_{i}}{M}$ and $\pi\frac{b_{i}}{M}$. Then the resulting drawing is an arc $k$-visibility representation of $G$.

Alternatively there exist arc $k$-visibility graphs which are not bar $k$-visibility graphs. The graph $K_{5}$ is not planar, so it is not a bar $0$-visibility graph. However $K_{5}$ has an arc visibility representation. For example consider the sightlines between the five arcs having endpoints $(1, 0)$ and $(1, \frac{\pi}{2})$, $(2, \frac{\pi}{6})$ and $(2, \frac{2\pi}{3})$, $(3, \frac{\pi}{4})$ and $(3, \frac{5\pi}{4})$, $(4, \frac{\pi}{3})$ and $(4, \frac{7\pi}{4})$, and $(5, 0)$ and $(5, \frac{\pi}{2})$. Since each pair of arcs is radially visible in this representation, then $K_{5}$ is an arc $0$-visibility graph. 

When considering arc and circle $k$-visibility graphs, we make two assumptions.

\begin{asmp}
If two endpoints of two arcs have the same angular coordinate, then we can move one slightly without deleting any edges in the arc $k$-visibility graph. So we assume that no two arcs have endpoints with the same angular coordinate since we are maximizing the number of edges.
\end{asmp}

\begin{asmp}
If there are two arcs that are the same distance from the origin, then we can slightly increase the radius of one so that their radii are different without affecting the arc $k$-visibility graph. Therefore we also assume that no two arcs are the same distance away from the origin. We then label the arcs with $a_1$, $a_2$, $\dots$, and $a_n$, where $a_i$ is given to the arc with the $i^{th}$ greatest radius.
\end{asmp}

Any circle can be turned into an arc without deleting any edges, but this could possibly add edges. As we are interested in upper bounds on the number of edges and the chromatic number of arc and circle $k$-visibility graphs, then we will prove bounds for arc $k$-visibility graphs which will also hold for circle $k$-visibility graphs. Therefore we can assume that arc or circle $k$-visibility representations only contain arcs.

In the next section we prove an upper bound of $(k+1)(3n-k-2)$ for $n \geq 4k+4$ on the number of edges and an upper bound of $6k+6$ on the chromatic number of arc and circle $k$-visibility graphs. Since $K_{4k+4}$ is a bar $k$-visibility graph \cite{AlDeank}, then the maximum number of edges in any arc or circle $k$-visibility graph on $n \leq 4k+4$ vertices is ${n \choose 2}$.

\section{Bounds on edges, chromatic number, and thickness}\label{sec:upperbounds}
Dean \emph{et al.} \cite{AlDeank} showed that the thickness of any bar $k$-visibility graph $G_k$ is at most $2k(9k-1)$ by coloring the edges based on a vertex-coloring of $G_{k-1}$. The following lemma provides a lower bound on the maximal thickness of bar $k$-visibility graphs.

\begin{lem}
\label{existbar}
There exist bar $k$-visibility graphs with thickness at least $k+1$ for all $k\geq 0$.
\end{lem}

\begin{proof}
Consider $m$ disjoint planar subgraphs of a bar $k$-visibility graph $G_k$ with $n$ vertices. It is a well known fact that the number of edges in a planar graph is at most six less than three times the number of vertices, so it follows that the number of edges in $G_k$ is at most $m(3n-6)$. Hartke \emph{et al.} \cite{Hartke} showed that if $G_k$ has $n \geq 4k+4$ vertices, then $G_k$ has at most $(k+1)(3n-4k-6)$ edges. Dean \emph{et al.} \cite{AlDeank} showed that this bound is sharp, so we consider a bar $k$-visibility graph with $(k+1)(3n-4k-6)$ edges. Therefore $m(3n-6)\geq(k+1)(3n-4k-6)$. It then follows that $\Theta(G_k)\geq (k+1)\frac{3n-4k-6}{3n-6}$. Fix $k$ and choose $n > \frac{4k^2+4k+6}{3}$ so that $(k+1)\frac{3n-4k-6}{3n-6} > k$. Then $\Theta(G_k)\geq k+1$.
\end{proof}

To bound the thickness of semi-bar $k$-visibility graphs, we define a one-bend construction as in \cite{Felsner}. We will assume no pair of semi-bars have the same length since we can change the lengths without deleting edges in the $k$-visibility graph.

\begin{defn}
The \emph{underlying semi-bar $k$-visibility graph} $G_k$ of $S$ is the graph with semi-bar $k$-visibility representation $S$. 
\end{defn}

Consider a semi-bar $k$-visibility representation $S$ of $G_{k}$ in which semi-bars are horizontal with all right endpoints on the same vertical line and all left endpoints on different vertical lines. We construct a \emph{one-bend drawing} from $S$ as Felsner and Massow did in \cite{Felsner}. In this drawing each edge consists of two segments connected at an endpoint.

To create a one-bend drawing of $G_k$, first widen the bars so that they are rectangles while keeping their lengths constant. Each vertex $v$ in the graph now corresponds to a rectangle $R_v$. Next draw each vertex on the midpoint of the left side of each rectangle. Then take the leftmost endpoint, say $u$, of each edge $e=\{u,v\}$. 

Then project $v$ orthogonally onto the nearest side of $R_u$, and call this projection $v'$. Note that the line between $v$ and $v'$ is a sightline between $R_u$ and $R_v$. Choose $\epsilon$ to be less than the minimum distance between any two endpoints of bars in $S$. Take $v_{\epsilon}$ on the side of $R_u$ containing $v'$ so that the length of $\overline{v'v_{\epsilon}}$ is $\epsilon$. Let $e$ be the union of the two line segments $\overline{uv_{\epsilon}}$ and $\overline{v_{\epsilon}v}$. Note that if two edges are adjacent, then by definition they will not intersect anywhere other than their common endpoint. See Figure~\ref{1b}.

\begin{figure}[h]
\begin{center}
\includegraphics[scale=0.5]{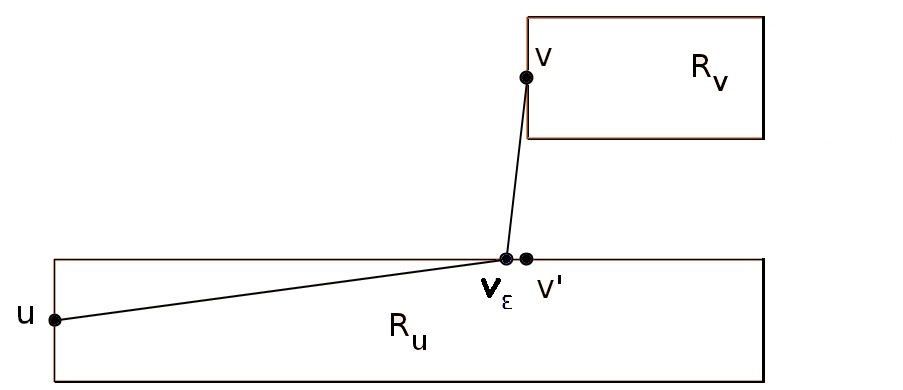}
\setlength{\abovecaptionskip}{0pt}
\caption{A one-bend drawing.}
\label{1b}
\end{center}
\end{figure}

\begin{thm}
\label{sbt}
If $G$ is a semi-bar $k$-visibility graph with $k\geq 1$, then $\Theta(G)\leq 2k$.
\end{thm}

\begin{proof}
The main idea of this proof is the next lemma:

\begin{lem}
For any bar $B$ in a semi-bar $k$-visibility representation, there are at most $2k-1$ longer bars in the representation with edges crossing $B$.
\end{lem}

\emph{Proof.} If there are $k+1$ or more longer bars on one side of $B$ with edges crossing $B$, then the bar farthest from $B$ on that side would have an edge crossing at least $k+1$ bars ($B$ and the $k$ bars on that side closest to $B$), a contradiction of $k$-visibility. Therefore there are at most $k$ longer bars on each side of $B$ with edges crossing $B$.

Assume that there are $k$ longer bars on each side of $B$ with edges crossing $B$. Now consider the top-most and bottom-most bars $B_1$ and $B_2$ respectively among those $2k$ bars. From our assumption there must be bars $b_1$ and $b_2$ on the lower and upper sides of $B$ respectively which are shorter than $B$ such that $B_{1}$ has an edge with $b_{1}$ and $B_{2}$ has an edge with $b_{2}$. Assume without loss of generality that $b_1$ is shorter than $b_2$. This implies that the edge from $B_1$ to $b_{1}$ must cross $b_2$, so this edge crosses $k+1$ bars. This is a contradiction, which implies that there are not $k$ longer bars on each side of $B$ with edges crossing $B$. This completes the proof. $\triangle$

Given a one-bend drawing of $G$, start coloring the bars and any edges connected to them in decreasing order of bar length using $2k$ colors. We can use $2k-1$ colors to color each of the longest $2k-1$ bars. For each bar, color its previously uncolored edges with the same color assigned to the bar. Suppose at least $i$ bars have been colored for $i \geq 2k-1$. Edges from longer bars colored with at most $2k-1$ colors will cross the $(i+1)^{st}$ longest bar, so we color this bar with a remaining color. Intersections only happen within bars, so the final coloring of edges produces $2k$ planar subgraphs. This completes the proof of Theorem~\ref{sbt}.\end{proof}

We now prove a lower bound on the maximum thickness of semi-bar $k$-visibility graphs.

\begin{thm}
\label{existsbar}
There exist semi-bar $k$-visibility graphs with thickness at least $\left\lceil \frac{2}{3}(k+1)\right\rceil$ for all $k\geq 0$.
\end{thm}

The proof of Theorem~\ref{existsbar} is analogous to the proof of Theorem~\ref{existbar}; the only difference is the sharp upper bound on the number of edges in a semi-bar $k$-visibility graph with $n$ vertices, which Felsner and Massow \cite{Felsner} proved was $(k+1)(2n-2k-3)$ for $n \geq 2k+2$. 

\subsection{Arc and circle $k$-visibility graphs}
Dean \emph{et al.} found upper bounds on the number of edges and the chromatic number of bar $k$-visibility graphs \cite{AlDeank}. Here we set upper bounds on these properties for arc, circle, and semi-arc $k$-visibility graphs.

It will suffice to find an upper bound on the number of edges of arc $k$-visibility graphs since any circles can be turned into arcs without deleting any edges. Consider an edge $\{u,v\}$ in the arc $k$-visibility graph, and let $U$ and $V$ be the arcs corresponding to $u$ and $v$. 

We define the \textit{angular coordinate of a line of sight} in the following manner. Previously we assigned unique angular coordinates to the points on each arc. Now we can associate each line of sight with the smallest angular coordinate of its two endpoints. We will call this number the angular coordinate of the line of sight. Note that the angular coordinate of a line of sight can vary between $0$ and $3\pi$.

Let $\ell(\{u,v\})$ denote the radial line segment between $U$ and $V$ whose angular coordinate is the infimum of the angular coordinates of all lines of sight between $U$ and $V$. If $\ell(\{u,v\})$ contains the negative endpoint of $U$ (respectively $V$) then we call $\{u,v\}$ a \emph{negative edge} of $U$ (respectively $V$).

If $\ell(\{u,v\})$ does not contain the negative endpoint of $U$ or $V$, then it must contain the positive endpoint of some arc $B$ that blocks the $k$-visibility between $U$ and $V$ before the endpoint. In this case we call $\ell(\{u,v\})$ a \emph{positive edge} of $B$.

\begin{lem}
\label{arcrem}
By definition, there are at most $k+1$ positive edges and at most $2k+2$ negative edges corresponding to each arc. Therefore, there are at most $(3k+3)n$ edges in a circle or arc $k$-visibility graph with $n$ vertices.
\end{lem}

\begin{thm}
In a circle or arc $k$-visibility graph with $n$ vertices, there are at most $(k+1)(3n-k-2)$ edges for $n > 4k+4$ and ${n \choose 2}$ edges for $n \leq 4k+4$.
\end{thm}

\begin{proof}
Suppose that $n > 4k+4$, since the maximum number of edges is ${n \choose 2}$ for $n \leq 4k+4$. Since we may assume that the circle $k$-visibility graph is an arc $k$-visibility graph, then name the arcs $a_{1}, a_{2}, \ldots, a_{n}$ in increasing order of distance from the center of the circle.

Lemma~\ref{arcrem} gives an upper bound of $3n(k+1)$ edges. However arcs $a_{n}$, $a_{n-1}$, $\ldots$, $a_{n-k}$ have at most $k+1$, $k+2$, $\ldots$, $2k+1$ negative edges respectively and $0$, $1$, $\ldots$, $k$ positive edges respectively. Therefore the upper bound on edges can be improved to
$$(3k+3)n-2\sum_{i=1}^{k+1}i=(k+1)(3n-k-2).$$
\end{proof}

\begin{rem}
Note that letting $k=0$ produces an upper bound of $3n-2$ on the number of edges that a circle or arc visibility graph can have.
\end{rem}

These upper bounds on the number of edges of circle $k$-visibility graphs give an upper bound on their chromatic numbers.

\begin{cor}
If $G$ is a circle $k$-visibility graph, then $\chi(G)\leq 6k+6$.
\end{cor}

The proof of the last corollary is nearly identical to the proof for bar $k$-visibility graphs by Dean \emph{et al.} since every subgraph of $G$ has a vertex of degree at most $6k+5$ \cite{AlDeank}.

\begin{thm}
In semi-arc $k$-visibility graphs on $n$ vertices, the maximum number of edges is ${n \choose 2}$ for $n \leq 3k+3$ and at most $(k+1)(2n-\frac{k+2}{2})$ for $n > 3k+3$.
\end{thm}

\begin{proof}
If $n \leq 3k+3$, then $K_{n}$ is a semi-arc $k$-visibility graph (see Figure \ref{sakv}). If $n > 3k+3$, then let $G$ be a graph on $n$ vertices with semi-arc $k$-visibility representation $S$. Every edge in $G$ can be drawn as a visibility segment in $S$ intersecting the positive endpoint of at least one of the semi-arcs in the edge. Since edges in the representation can cross at most $k$ arcs that are not in the edge, then at most $2k+2$ edges can be drawn intersecting the positive endpoint of each arc. However the $k+1$ outermost arcs have at most $k+1, k+2, \ldots, 2k+1$ edges respectively that can be drawn intersecting their positive endpoint, which implies the upper bound.
\end{proof}

\begin{figure}[h]
\begin{center}
\includegraphics[scale=1]{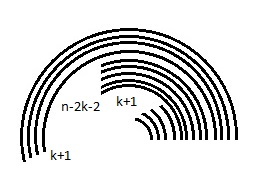}
\setlength{\abovecaptionskip}{0pt}
\caption{A semi-arc $k$-visibility representation with $n$ vertices and $(k+1)(2n-\frac{3k+6}{2})$ edges for $n \geq 3k+3$.}
\label{sakv}
\end{center}
\end{figure}

\begin{cor}
If $G$ is a semi-arc $k$-visibility graph, then $\chi(G)\leq 4k+4$.
\end{cor}

Again, the proof of the last corollary is like the proof for bar $k$-visibility graphs since every subgraph of $G$ has a vertex of degree at most $4k+3$ \cite{AlDeank}. The construction in Figure \ref{sakv} implies the next lower bound.

\begin{thm}\label{sakvl}
The maximum number of edges in a semi-arc $k$-visibility graph is at least $(k+1)(2n-\frac{3k+6}{2})$ for $n \geq 3k+3$.
\end{thm}

\section{Counting edges in semi-bar $k$-visibility graphs}\label{sec:semibar}
In this section we derive a formula for counting the number of edges in semi-bar $k$-visibility graphs. For the final formula semi-bars are allowed to have the same lengths, but we first consider the case when the lengths are different. 

Let $G = (V, E)$ be a semi-bar $0$-visibility graph with $n$ vertices. Then $G$ has some semi-bar visibility representation $S_{G} = \left\{s_{v}\right\}_{v \in V}$ of disjoint horizontal segments with left endpoints on the $y$-axis (semi-bars) such that for all $a, b \in V$ all semi-bars between $s_{a}$ and $s_{b}$ are shorter than both $s_{a}$ and $s_{b}$ if and only if $\left\{a,b\right\} \in E$.

Let the function $A(S)$ be the number of semi-bars in $S$ which are taller than all semi-bars above them, and $U(S)$ be the number of semi-bars in $S$ which are taller than all semi-bars under them. These are analogous to the numbers above and below each column in skyscraper puzzles. Moreover if all semi-bars have different lengths, then $S_{G}$ corresponds to a permutation of the integers $\left\{1, \ldots, n \right\}$ with the topmost bar of $S_{G}$ representing the first term of the permutation. The function $A(S)$ corresponds to the number of left-to-right maxima in the permutation, while $U(S)$ corresponds to the number of right-to-left maxima in the permutation.

For each $s \in S$ let $a(s) = 1$ if $s$ is taller than all semi-bars above it and let $a(s) = 0$ otherwise. So $a(s) = 1$ precisely when the term in the permutation corresponding to $s$ is a left-to-right maximum. Let $u(s) = 1$ if $s$ is taller than all semi-bars under it and let $u(s) = 0$ otherwise. Then $A(S) = \sum_{s \in S} a(s)$ and $U(S) = \sum_{s \in S} u(s)$.

\begin{lem}
\label{skyscraper}
If $S_{G}$ is any semi-bar visibility representation of $G$ and all semi-bars in $S_{G}$ have different lengths, then the number of edges in $G$ is $2n-A(S_{G})-U(S_{G})$.
\end{lem}

\begin{proof}
Pick an arbitrary semi-bar visibility representation $S_{G}$ of $G$. For each $v \in V$, count how many edges in $E$ include $v$ and some $w$ for which $s_{w}$ is taller than $s_{v}$. Then each $v$ contributes $2-a(s_{v})-u(s_{v})$ edges, so there are $2n-A(S_{G})-U(S_{G})$ total edges.
\end{proof}

We now extend this formula to all semi-bar $k$-visibility graphs. Let $G_k = (V, E)$ be a semi-bar $k$-visibility graph. Then $G_k$ has a semi-bar $k$-visibility representation $S_{G_k} = \left\{s_{v}\right\}_{v \in V}$ of disjoint horizontal semi-bars with left endpoints on the $y$-axis such that for all $a, b \in V$, all but at most $k$ semi-bars between $s_{a}$ and $s_{b}$ are shorter than both $s_{a}$ and $s_{b}$ if and only if $\left\{a,b\right\} \in E$. Define $\left\{s_{a}, s_{b}\right\}$ to be a $j$-visibility edge if all but exactly $j$ semi-bars between $s_{a}$ and $s_{b}$ are shorter than both $s_{a}$ and $s_{b}$.

Let the function $A_{j}(S)$ be the number of semi-bars in $S$ which are taller than all but at most $j$ semi-bars above them, and $U_{j}(S)$ be the number of semi-bars in $S$ which are taller than all but at most $j$ semi-bars under them. For each $s \in S$ let $a_{j}(s) = 1$ if $s$ is taller than all but at most $j$ semi-bars above it and let $a_{j}(s) = 0$ otherwise. Let $u_{j}(s) = 1$ if $s$ is taller than all but at most $j$ semi-bars under it and let $u_{j}(s) = 0$ otherwise. Then $A_{j}(S) = \sum_{s \in S} a_{j}(s)$ and $U_{j}(S) = \sum_{s \in S} u_{j}(s)$.

Call an unordered pair of semi-bars $\left\{s_{a}, s_{b}\right\}$ a $j$-bridge if $s_{a}$ is the same height as $s_{b}$ and all but exactly $j$ semi-bars between $s_{a}$ and $s_{b}$ are shorter than $s_{a}$. A semi-bar can be contained in at most two $j$-bridges for each $j$. Let $Br_{j}(S)$ denote the number of $j$-bridges in $S$. 

\begin{thm}
\label{edgecount}
If $S_{G_k}$ is any semi-bar $k$-visibility representation of $G_k$, then the number of edges in $G_k$ is $$2(k+1)n-\displaystyle\sum_{j = 0}^{k} (A_{j}(S_{G_k})+U_{j}(S_{G_k})+Br_{j}(S_{G_k})).$$
\end{thm}

\begin{proof}
Pick an arbitrary semi-bar $k$-visibility representation $S_{G_k}$ of $G_k$. Fix $j \leq k$, and for each $v \in V$, count how many $j$-visibility edges include $s_{v}$ and some $s_{w}$ for which $s_{w}$ is at least as tall as $s_{v}$. Then each $v$ contributes $2-a_{j}(s_{v})-u_{j}(s_{v})$ $j$-visibility edges, but the $j$-visibility edge between $s_{a}$ and $s_{b}$ is double counted whenever $\left\{s_{a}, s_{b}\right\}$ is a $j$-bridge. So there are $2n-A_{j}(S_{G_k})-U_{j}(S_{G_k})-Br_{j}(S_{G_k})$ total $j$-visibility edges in $S_{G_k}$. Then there are $2(k+1)n-\sum_{j = 0}^{k} (A_{j}(S_{G_k})+U_{j}(S_{G_k})+Br_{j}(S_{G_k}))$ total edges in $G_k$.
\end{proof}

Since Felsner and Massow showed a tight upper bound of $(k+1)(2n-2k-3)$ on the number of edges in semi-bar $k$-visibility graphs with $n \geq 2k+2$ vertices, then Theorem~\ref{edgecount} implies the next corollary.

\begin{cor}
If $S_{G_k}$ is any semi-bar $k$-visibility representation of $G$ with $n \geq 2k+2$ vertices, then $$\sum_{j = 0}^{k} (A_{j}(S_{G_k})+U_{j}(S_{G_k})+Br_{j}(S_{G_k})) \geq (k+1)(2k+3).$$
\end{cor}

\section{Open Problems}

The results in this paper leave open questions beyond the ones mentioned in \cite{AlDeank, Felsner, Hartke}. 

\begin{ques}
What is the maximum number of edges in a semi-arc $k$-visibility graph on $n$ vertices for $n > 3k+3$?
\end{ques}

We conjecture that the bound in Theorem \ref{sakvl} is tight.

\begin{conj}\label{maxsakv}
The maximum number of edges in a semi-arc $k$-visibility graph on $n$ vertices is $(k+1)(2n-\frac{3k+6}{2})$ for $n \geq 3k+3$.
\end{conj}

This conjecture would also imply the next conjecture.

\begin{conj}
$K_{3k+4}$ is not a semi-arc $k$-visibility graph.
\end{conj}

\begin{proof}
By Conjecture \ref{maxsakv}, a semi-arc $k$-visibility graph on $3k+4$ vertices can have at most $\frac{1}{2}(k+1)(9k+10)$ edges, which is less than $\frac{1}{2}(3k+3)(3k+4)$.
\end{proof}

There are also similar open questions about arc $k$-visibility graphs.

\begin{ques}
What is the maximum number of edges in an arc $k$-visibility graph on $n > 4k+4$ vertices?
\end{ques}

\begin{ques}
What is the largest complete arc $k$-visibility graph?
\end{ques}

\section{Acknowledgments} 
We thank Jacob Fox for suggesting this project. This paper is based on a Research Science Institute math project where the first author was a student, the second author was a graduate student mentor, and the third author was the head mentor in math. The first author would like to thank Dr.\ John Rickert for his excellent advice on research; Mr.\ Timothy J. Regan from the Corning Incorporated Foundation, Mr.\ Peter L. Beebee, and Mr.\ David Cheng for their sponsorship; and Mr.\ Zachary Lemnios, Dr.\ Laura Adolfie, Dr.\ John Fischer, and Dr.\ Robin Staffin from the Department of Defense for naming him as a Department of Defense Scholar. He would also like to thank the Center for Excellence in Education, the Research Science Institute, and the Massachusetts Institute of Technology for making this endeavor possible. The second author was supported by the NSF Graduate Research Fellowship under Grant No. 1122374.

\end{document}